\pdfoutput=1 

\documentclass[12pt]{amsart}

\usepackage[style=alphabetic, maxnames=99]{biblatex}
\AtBeginBibliography{\small}
\usepackage{microtype}
\microtypesetup{protrusion=false}
\usepackage{xpatch}
\usepackage[english]{babel}
\usepackage{amsfonts}
\usepackage{amsmath}
\usepackage{amssymb}
\usepackage{mathrsfs}
\usepackage{stmaryrd}

\usepackage{amsthm}
\xpatchcmd{\proof}{\itshape}{\bfseries}{}{}
\usepackage{thmtools}
\usepackage{thm-restate}
\usepackage{tikz}
\usepackage{tikz-cd}
\usetikzlibrary{arrows.meta,automata,positioning}
\usetikzlibrary{svg.path}
\usepackage{amscd}
\usepackage{pdfpages}
\usepackage{float}
\usepackage{enumitem}
\usepackage[labelfont=bf]{caption}
\usepackage[linewidth=0pt]{mdframed}
\usepackage{subfig}
\usepackage{makecell}
\usepackage{multicol}
\usepackage{listings}
\usepackage{fancyhdr}
\usepackage[bottom, perpage]{footmisc}
\usepackage{etoolbox}

\patchcmd{\section}{\scshape}{\bfseries}{}{}
\makeatletter
\renewcommand{\@secnumfont}{\bfseries}
\makeatother

\allowdisplaybreaks

\usepackage{xr}
\usepackage{subfiles}

\usepackage{setspace}
\usepackage{xstring}
\usepackage{hyperref}
\usepackage[capitalize]{cleveref}
\usepackage{url}
\hypersetup{
	colorlinks   = true, 
	urlcolor     = cyan, 
	linkcolor    = blue, 
	citecolor   = blue,
	breaklinks = true
}

\numberwithin{equation}{section}
\numberwithin{figure}{section}
\makeatletter
\let\c@equation\c@figure
\makeatother

\makeatletter
\newcommand{\labeltext}[3][]{%
	\@bsphack%
	\csname phantomsection\endcsname
	\def\tst{#1}%
	\def\labelmarkup{}
	\def\refmarkup{}%
	\ifx\tst\empty\def\@currentlabel{\refmarkup{#2}}{\label{#3}}%
	\else\def\@currentlabel{\refmarkup{#1}}{\label{#3}}\fi%
	\@esphack%
	\labelmarkup{#2}
}
\makeatother

\declaretheorem[numberwithin=section,numberlike=equation,style=plain]{theorem}
\declaretheorem[numbered=no,style=plain,name=Theorem]{theorem*}
\declaretheorem[numberwithin=section,numberlike=equation,style=plain]{proposition}
\declaretheorem[numbered=no,style=plain,name=Proposition]{proposition*}
\declaretheorem[numberwithin=section,numberlike=equation,style=plain]{lemma}
\declaretheorem[numbered=no,style=plain,name=Lemma]{lemma*}
\declaretheorem[numberwithin=section,numberlike=equation,style=plain]{corollary}
\declaretheorem[numbered=no,style=plain,name=Corollary]{corollary*}
\declaretheorem[numberwithin=section,numberlike=equation,style=plain]{conjecture}
\declaretheorem[numbered=no,style=plain,name=Conjecture]{conjecture*}

\declaretheorem[numbered=no,style=plain,name=Question]{question*}

\declaretheorem[numberwithin=section,numberlike=equation,style=definition]{definition}
\declaretheorem[numbered=no,style=definition,name=Definition]{definition*}
\declaretheorem[numberwithin=section,numberlike=equation,style=definition]{remark}
\declaretheorem[numbered=no,style=definition,name=Remark]{remark*}

\declaretheorem[numbered=no,style=definition,name=Notation]{notation*}

\declaretheorem[numbered=no,style=definition,name=Axiom]{axiom*}

\declaretheorem[numbered=no,style=definition,name=Construction]{construction*}

\declaretheorem[numbered=no,style=definition,name=Algorithm]{algorithm*}

\declaretheorem[numbered=no,style=definition,name=Summary]{summary*}

\declaretheorem[numbered=no,style=definition,name=Property]{property*}

\declaretheorem[numbered=no,style=definition,name=Note]{note*}

\declaretheorem[numbered=no,style=definition,name=Example,qed=$\diamondsuit$]{example*}
\usepackage{fullpage}
\usepackage[all]{xy}
\usepackage{aliascnt}
\usepackage{combelow}

\newcommand{\bR}{\mathbb{R}}

\newcommand{\bZ}{\mathbb{Z}}
\newcommand{\R}{\mathbb{R}}

\newcommand{\Z}{\mathbb{Z}}
\newcommand{\F}{\mathbb{F}}

\DeclareMathOperator{\ent}{ent}

\title{There are no good infinite families of toric codes}

\author{Jason P.~Bell, Sean Monahan, Matthew Satriano, Karen Situ, and Zheng Xie}

\address{Jason P.~Bell, Department of Pure Mathematics, University of Waterloo}
\email{jpbell@uwaterloo.ca}

\address{Sean Monahan, Department of Pure Mathematics, University of Waterloo}
\email{sean.monahan@uwaterloo.ca}

\address{Matthew Satriano, Department of Pure Mathematics, University of Waterloo}
\email{msatriano@uwaterloo.ca}

\address{Karen Situ, Department of Pure Mathematics, University of Waterloo}
\email{k4situ@uwaterloo.ca}

\address{Zheng Xie, Department of Pure Mathematics, University of Waterloo}
\email{z6xie@uwaterloo.ca}


\begin{document}

\begin{abstract}
	Soprunov and Soprunova posed a question on the existence of infinite families of toric codes that are ``good" in a precise sense.   
	We prove that such good families do not exist by proving a more general Szemer\'edi-type result:~for all $c\in(0,1]$ and all positive integers $N$, subsets of density at least $c$ in $\{0,1,\dots,N-1\}^n$ contain hypercubes of arbitrarily large dimension as $n$ grows.
\end{abstract}

\maketitle

\section{Introduction}

After the pioneering work of Hamming in the 1940s and 50s \cite{hamming1950error}, error-correcting codes have been of central importance for our modern technologies, and have had wide connections to theoretical areas of mathematics such as number theory and algebraic geometry. Following the foundational work of Goppa \cite{goppa1981codes-curves,goppa1982algebraic-codes} on algebraic-geometric codes -- extending the famous Reed-Solomon code \cite{reed1960polynomial-codes} -- Tsfasman and Vl\u{a}du\c{t} \cite{tsfasman1991algebraic-codeds} established a framework for constructing error-correcting codes from algebraic varieties. Applying this framework to the class of toric varieties has seen great success due to the combinatorial nature of these objects which is well-suited for performing explicit computations.  

Toric codes were first introduced by Hansen \cite{hansen2000toric-surfaces,hansen2002toric-hirzebruch} and have subsequently been studied by a number of authors, including Joyner, Little, Schenck, Schwarz, Ruano, Soprunov, and Soprunova \cite{joyner2004toric-codes,little-schenck2006toric-codes,little-schwarz2007toric-codes,ruano2007toric-codes,soprunov-soprunova2009toric-codes-minkowski,soprunov-soprunova2010bringing-toric-codes}. 
These codes are built from the following data. Fix a prime power $q$, which determines a finite field $\F_q$, and let $P$ be an integral convex polytope which is contained in the hypercube $[0,q-2]^n$ in $\R^n$. One can construct the \textit{toric code} $C_P:=C_P(\F_q)$ associated to $P$ as the image of an explicit injective linear map $\F_q\{P\cap \Z^n\}\to \F_q^{(q-1)^n}$; see, for example, \cite[Section 2]{dolorfino2022good-families-toric} for more details. For simplicity, we often refer to just the polytope $P$ rather than the code $C_P$. 

There are three key quantities associated to an error-correcting code, which have pleasant descriptions in the case of the toric code $C_P$: the \textit{block length}, which is simply $(q-1)^n$; the \textit{dimension}, which is the number of lattice points in $P$, i.e. $|P\cap\Z^n|$; and the \textit{minimum distance}, which is the minimum Hamming distance over all nonzero vectors in $C_P$ (see \cite[Definition 2.4]{dolorfino2022good-families-toric}). 

To determine if a toric code is ``good'' from a coding-theoretic perspective, one considers the relationships between these three quantities. In particular, one wants both the dimension and the minimum distance to be large relative to the block length. That is, one wants the \textit{relative minimum distance} $$d(P):=\frac{\text{minimum distance of } C_P}{(q-1)^n}$$ and the \textit{information rate} $$R(P):=\frac{|P\cap\Z^n|}{(q-1)^n}$$ to be large.\footnote{Note that we use $d(P)$ for the relative minimum distance, even though some sources such as \cite{dolorfino2022good-families-toric} use $d(C_P)$ to denote the minimum distance and $\delta(C_P)$ for the relative minimum distance. We prefer to reserve $\delta$ for the density of a subset.} The difficulty in finding good toric codes is that $d(P)$ and $R(P)$ are inversely related:~as one becomes larger, the other becomes small. 

In \cite{soprunov-soprunova2010bringing-toric-codes}, Soprunov and Soprunova consider infinite families of toric codes with a fixed value of $q$. Since they do not write down a formal definition, we refer to \cite{dolorfino2022good-families-toric}. Recall that $q$ is a fixed prime power. 

\begin{definition}[{\cite[Definition 3.1]{dolorfino2022good-families-toric}}]\label{def:infinite family toric codes}
	An \textit{infinite family of toric codes} is a sequence $\{P_i\}_i$ of nonempty integral convex polytopes satisfying $P_i\subseteq [0,q-2]^{n_i}\subseteq \R^{n_i}$ such that $n_i\to\infty$ as $i\to\infty$. 
\end{definition}

Note that Soprunov and Soprunova mainly consider the case where $n_i=i$, so \cref{def:infinite family toric codes} is more general. They are interested in infinite families of toric codes which are good in the sense mentioned above. 

\begin{definition}[{\cite[Section 4]{soprunov-soprunova2010bringing-toric-codes}, cf.~\cite[Definition 3.2]{dolorfino2022good-families-toric}}]\label{def:good family of toric codes}
	An infinite family of toric codes $\{P_i\}_i$ is called \textit{good} if both $d(P_i)$ and $R(P_i)$ approach positive constants as $i\to\infty$. 
\end{definition}

Soprunov and Soprunova showed that certain explicit constructions of toric codes did not produce good families, and they say, ``It would be interesting to find an infinite good family of toric codes". This motivated the work of Dolorfino et al. \cite{dolorfino2022good-families-toric} in which they conjecture that there are \textit{no} good infinite families of toric codes. To support this conjecture, they introduce the following statistic for polytopes. 

First, let us clarify that an \textit{integer affine transformation} $F:\R^m\to \R^n$ is of the form $F(x)=Ax+b$ where $A$ and $b$ have integer entries; and we say that $F$ is a \textit{unimodular affine transformation} if the columns of $A$ form part of a $\Z$-basis for $\Z^n$ (note that $A$ need not be square). 

\begin{definition}[{\cite[Definition 4.1]{dolorfino2022good-families-toric}}]\label{def:M(P)}
	Given an integral convex polytope $P$, let 
	\begin{align*}
		M(P):= \max\{~m~ \mid \exists\text{ a unimodular affine transformation $F$ such that } F([0,1]^m)\subseteq P\}.
	\end{align*}
	If $P=\varnothing$, then $M(P):=-\infty$. 
\end{definition}

That is, $M(P)$ measures the dimension of the largest unit hypercube contained in $P$. The expectation is that, if $M(P)$ is large, then the number of lattice points in $P$ should be large. In \cite[Proposition 4.2]{dolorfino2022good-families-toric} they show that, for any infinite family of toric codes $\{P_i\}_i$, if the sequence $\{M(P_i)\}_i$ is unbounded, then $\{d(P_i)\}_i$ cannot converge to a positive constant. In particular, this implies that if $\{M(P_i)\}_i$ is unbounded, then $\{P_i\}_i$ is \textit{not} a good family. Thus, to determine if there are good infinite families of toric codes, one may restrict to the case where $\{M(P_i)\}_i$ is bounded. In this case, Dolorfino et al. make the following conjecture, which would imply that $\{P_i\}_i$ is \textit{not} a good family in the case where $\{M(P_i)\}_i$ is bounded. 

\begin{conjecture}[{\cite[Conjecture 4.3]{dolorfino2022good-families-toric}}]\label{conj:main}
	If $\{P_i\}_i$ is an infinite family of toric codes such that $\{M(P_i)\}_i$ is bounded, then $R(P_i)\to 0$ as $i\to\infty$. 
\end{conjecture}

Our first main theorem addresses this conjecture.

\begin{theorem}\label{thm:no good families of toric codes}
	\cref{conj:main} is true. Therefore, there are no good infinite families of toric codes (in the sense of \cref{def:good family of toric codes}).
\end{theorem}

In fact, \cref{thm:no good families of toric codes} follows from a more general theorem regarding subsets of large density, which we now describe. Given an inclusion of finite sets $X\subseteq Y$, we let $\delta_Y(X):=\frac{|X|}{|Y|}$ be the \emph{density} of $X$ in $Y$; frequently $Y$ is understood from context, in which case $\delta_Y(X)$ is written simply as $\delta(X)$. 

There is a general principle, central both in additive combinatorics and ergodic theory, that sufficiently dense sets should have structure. A well-known instance of this philosophy is Szemer\'edi's Theorem \cite{szemeredi1975integers} which shows that dense subsets in $\bZ$ have arithmetic progressions. In the context of toric codes, the information rate $R(P)$ is precisely the density of $P\cap\Z^n$ in $\{0,1,\ldots,q-2\}^n$.

To state our main result, we require the following adaptation of \cref{def:M(P)} to finite sets of lattice points. Note that the following uses injective integer affine transformations, which is more general than unimodular affine transformations as in \cref{def:M(P)}. 

\begin{definition}\label{def:M_0(S)}
	Given a finite set $S\subseteq \Z^n$, let
	\begin{align*}
		M_0(S) := \max\left\{ ~m~ ~\middle|~~ \begin{array}{c} \exists \text{ an injective integer affine transformation} \\ F \text{ such that } F(\{0,1\}^m)\subseteq S \end{array} \right\}.
	\end{align*}
	If $S=\varnothing$, then $M_0(S):=-\infty$. 
\end{definition}

\begin{remark}\label{rmk:compare M and M_0}
	Given an integral convex polytope $P\subseteq \R^n$, we have $M(P)\leq M_0(P\cap \Z^n)$. 
\end{remark}

In general, $M(P)$ does not necessarily equal $M_0(P\cap \Z^n)$. However, the following proposition gives an important relationship between these two statistics which is key in proving \cref{thm:no good families of toric codes}. 

\begin{proposition}\label{prop:bounded M and M_0}
	Let $\{P_i\}_i$ be an infinite family of toric codes. Then $\{M(P_i)\}_i$ is bounded if and only if $\{M_0(P_i\cap \Z^{n_i})\}_i$ is bounded.
\end{proposition}

The following theorem is our main result. Notice that \cref{thm:main} holds for arbitrary finite sets, not just sets of polytope lattice points, and we are able to replace the prime power $q$ with any integer $N\geq 2$. 

\begin{theorem}\label{thm:main}
	Fix $N\geq2$ and $c\in(0,1)$, and let
	\[
	f_N(n,c):=\inf\{M_0(S)\mid S\subseteq \{0,1,\dots,N-1\}^n,\ \delta(S)\geq c\}.
	\]
	Then
	\[
	\lim_{n\to\infty} \frac{f_N(n,c)}{\log_2 n}=1.
	\]
\end{theorem}

\begin{remark}\label{rmk:no good families from main thm}
	If $\{R(P_i)\}_i$ does not converge to $0$, then $\{M_0(P_i\cap\Z^{n_i})\}_i$ is unbounded by \cref{thm:main}, and thus $\{M(P_i)\}_i$ is unbounded by \cref{prop:bounded M and M_0}. In this way, we obtain \cref{thm:no good families of toric codes} as a consequence of \cref{thm:main}. In fact, \cref{thm:main} tells us that, if $\{R(P_i)\}_i$ converges to a value in $(0,1)$, the sequence $\{M_0(P_i\cap \Z^{n_i})\}_i$ grows at least logarithmically in $n_i$. 
\end{remark}

Our paper adds to the literature of Szemer\'edi-style results, proving that for $n\gg0$, dense subsets of $\{0,1,\dots,N-1\}^n$ must contain arbitrarily large hypercubes (via an injective integer affine transformation). 

Notice that \cref{thm:main} does not include the cases of $c=0$ or $c=1$. When $c=1$, we have $f_N(n,1)=n$ for all $n$. On the other hand, in the case when $c=0$, we show the following.

\begin{theorem}\label{thm:c=0}
	Let $N\ge 2$ be a positive integer and suppose that for each $n\ge 0$ we have a nonempty subset $S_n\subseteq \{0,1,\ldots , N-1\}^n$.  If 
	$$\limsup_{n\to\infty} \frac{\log_N(|S_n|)}{n} = 1,$$ then $\limsup_n M_0(S_n)=\infty$ as $n\to\infty$. Furthermore this is optimal in the following sense: given $\epsilon>0$, we can find a family of subsets $S_n\subseteq \{0,1,\ldots , N-1\}^n$ with $\{M_0(S_n)\}_n$ uniformly bounded and $\frac{\log_N(|S_n|)}{n}>1-\epsilon$ for every $n$. 
\end{theorem}

The quantity $\limsup_{i} \frac{\log_N(|P_i\cap \bZ^{n_i}|)}{n_i}$ is known as the \emph{entropy} $\ent(\{P_i\}_i)$; see
\cite[Equation (2)]{Ceccherini-Silberstein}. Thus, \cref{thm:c=0} yields the following corollary which shows that, even if one replaces the condition on $\{R(P_i)\}_i$ in \cref{def:good family of toric codes} with the condition that $\ent(\{P_i\}_i)=1$, then such infinite families of toric codes still do not exist.

\begin{corollary}
	Let $\{P_i\}_i$ be an infinite family of toric codes. If $\ent(\{P_i\}_i)=1$, then $d(P_i)\to0$ as $i\to\infty$.
\end{corollary}
\begin{proof}
	\cref{thm:c=0} shows that $\limsup_i M_0(P_i)=\infty$, and \cite[Proposition 4.2]{dolorfino2022good-families-toric} combined with \cref{prop:bounded M and M_0} imply $d(P_i)\to0$.
\end{proof}

The remainder of this paper is dedicated to proving \cref{thm:main} and \cref{thm:c=0}.

\section*{Acknowledgments}
It is a pleasure to thank Ajneet Dhillon, Brett Nasserden, Sayantan Roy-Chowdhury, Owen Sharpe, and Sophie Spirkl. We would like to especially thank Jim Geelen for enlightening conversations. Lastly, we want to thank the anonymous referees for their helpful feedback, and in particular their comments that inspired us to introduce \cref{def:M_0(S)}. 

The first and third-named authors were partially supported by Discovery Grants from the National Science and Engineering Research Council of Canada (NSERC); in addition, the third-named author was supported by a Mathematics Faculty Research Chair from the University of Waterloo. This paper is the outcome of an NSERC-USRA project; we thank NSERC for their support.

\section{Lower bound for the growth rate of $f_N$}

Throughout this paper, $N\geq2$ is fixed so we suppress it in the notation $f_N$. All logarithms without an explicit base are taken with base $N$. For any non-negative integer $k$, we let
\[
[k]:=\{0,1\dots,k-1\}.
\]
Our goal in this section is to give the following lower bound for $f$.

\begin{proposition}\label{prop:main-lower-bound}
	For all $c\in(0,1]$, 
	\[
	\liminf_{n\to\infty}\frac{f(n,c)}{\log_2(n)}\geq1.
	\]
\end{proposition}

We begin with a preliminary result which says that if we have sufficiently many sets of density at least $c$, then there must be a sufficiently large intersection.

\begin{lemma}\label{l:density-of-intersections}
	Suppose $c\in (0,1]$, $t\ge \frac{2}{c}$, and we have subsets $X_1,\ldots,X_t\subseteq[k]$ each of density at least $c$. Then there exist distinct $i$ and $j$ such that
	\[
	\delta(X_i\cap X_j)\geq \frac{2}{(\frac{2}{c}+1)^2}
	\]
\end{lemma}
\begin{proof}
	Shrinking our collection of sets, if necessary, we may assume $t=\lceil \frac{2}{c}\rceil$. Then the inclusion-exclusion principle says
	\[
	k\geq \sum_i|X_i| - \sum_{i<j}|X_i\cap X_j|.
	\]
	Dividing by $k$ and letting $m=\max_{i<j}\delta(X_i\cap X_j)$, we see
	\[
	m{t\choose 2}\geq\sum_i\delta(X_i)-1\geq tc-1.
	\]
	Since $tc\geq2$ and ${t\choose 2}\leq\frac{t^2}{2}\leq \frac{1}{2}(\frac{2}{c}+1)^2$, we see $m\geq \frac{2}{(\frac{2}{c}+1)^2}$, as desired.
\end{proof}

Next, we relate $M_0(S)$ for a given set $S$ to the $M_0$-value of a set living in a lower-dimensional space by considering elements in $S$ with a specified ``prefix". 

\begin{lemma}\label{l:MP-vs-MSa}
	Let $S\subseteq[N]^n$, $1\leq r<n$, and $a,b\in[N]^r$ distinct. We let
	\[
	T_x:=\{p\in[N]^{n-r}\mid (x,p)\in S\}
	\]
	for $x\in\{a,b\}$. Then
	\[
	M_0(S)\geq M_0(T_a\cap T_b)+1.
	\]
\end{lemma}
\begin{proof}
	First, if $T_a\cap T_b=\varnothing$, then the desired result holds because $M_0(S)\geq -\infty$, so we may assume going forward that $T_a\cap T_b\neq\varnothing$. Let $m=M_0(T_a\cap T_b)\geq 0$. By definition, we have an integer affine injection $\iota\colon\{0,1\}^m\to T_a\cap T_b$. We extend this to a map $\iota'\colon\{0,1\}^{m+1}\to S$ where $\iota'(0,y)=(a,\iota(y))$ and $\iota'(1,y)=(b,\iota(y))$. Note that $\iota'$ is an injection because $\iota$ is injective and $a\neq b$. Furthermore, $\iota'$ is an integer affine map; indeed, there exists an integer matrix $A$ and vector $z$ such that $\iota(y)=Ay+z$, hence
	\[
	\iota'(w,y)=
	\begin{bmatrix}
		b-a & 0 \\ 0 & A
	\end{bmatrix}
	\begin{bmatrix}
		w\\
		y
	\end{bmatrix}
	+ 
	\begin{bmatrix}
		a\\
		z
	\end{bmatrix}.
	\]
	As a result, $M_0(S)\geq m+1$.
\end{proof}

Combining the above two lemmas, we obtain an inductive lower bound on $f$.

\begin{proposition}\label{prop:inductive-f-bound}
	For $c\in(0,1]$ 
	and $n>\lceil\log(8c^{-2})\rceil$, we have
	\[
	f(n,c) \ge f\big(n-\lceil\log(8c^{-2})\rceil, 2c^2(c+4)^{-2}\big)+1.
	\]
\end{proposition}
\begin{proof}
	Let $S\subseteq[N]^n$ with $\delta(S)\geq c$. For each $1\leq r<n$, 
	let
	\[
	k_r := \left|\left\{a\in[N]^{r} \mid \delta(T_a) \geq \frac{c}{2}\right\}\right|
	\]
	with $T_a$ defined as in Lemma \ref{l:MP-vs-MSa}. Since $S=\coprod_a(a\times T_a)$, we see
	\begin{align*}
		c &\leq \delta(S)=\frac{1}{N^n}\sum_a|T_a|=\frac{1}{N^r}\sum_a\delta(T_a)\\
		&\leq \frac{k_r}{N^r} + \frac{N^r}{N^r}\cdot\frac{c}{2} = \frac{k_r}{N^r} + \frac{c}{2};
	\end{align*}
	the first inequality on the second line uses the bound $\delta(T_a)\leq1$ for all $a$ with $\delta(T_a)\geq \frac{c}{2}$, of which there are $k_r$, and uses the bound $\delta(T_a)\leq\frac{c}{2}$ for the remaining $a$, of which there are at most $N^r$.
	
	It follows that for $r\geq\log(8c^{-2})$, we have 
	\[
	k_r\geq\frac{c}{2}N^r\geq\frac{2}{c/2}.
	\]
	Hence, considering those $a$ with $\delta(T_a)\geq\frac{c}{2}$, Lemma \ref{l:density-of-intersections} tells us there exist $a\neq b$ with $\delta(T_a\cap T_b)\geq\frac{2c^2}{(c+4)^2}$. It follows then from Lemma \ref{l:MP-vs-MSa} that
	\[
	M_0(S)\geq M_0(T_a\cap T_b)+1\geq f(n-r,2c^2(c+4)^{-2})+1.
	\]
	In particular, we may take $r=\lceil\log(8c^{-2})\rceil$.
\end{proof}

We turn now to the proof of the main result of this section.

\begin{proof}[{Proof of Proposition \ref{prop:main-lower-bound}}]
	For all sufficiently small $\epsilon>0$, we have
	\begin{equation}\label{eqn:epsilon-suff-small-liminf}
		\frac{1}{1+\log_2(1+\frac{\epsilon}{6})}\geq 1-\frac{\epsilon}{2}
	\end{equation}
	since the Taylor expansion shows $(1+\log_2(1+\frac{x}{6}))^{-1}=1-\frac{x}{6\log_e 2} + O(x^2)$ as $x\to 0$. By definition, $f(n,c)\geq f(n,c')$ if $c'\leq c$. Thus, we may freely replace $c$ by a smaller value. In particular, we may assume
	\[
	\log c\leq \frac{3}{\epsilon}\min\left(\log\frac{2}{25}, -(1+\log 8)\right).
	\]
	Letting $\alpha:=2+\frac{\epsilon}{3}$, the above two inequalities imply
	\[
	\frac{2c^2}{(c+4)^2}\geq \frac{2c^2}{25}\geq c^\alpha\quad\textrm{and}\quad  n-\lceil\log(8c^{-2})\rceil\geq n+\alpha\log c.
	\]
	Then Proposition \ref{prop:inductive-f-bound} tells us if $n>\lceil\log(8c^{-2})\rceil$, then
	\[
	f(n,c) \ge f\big(n-\lceil\log(8c^{-2})\rceil, 2c^2(c+4)^{-2}\big)+1\ge f\big(n-\lceil\log(8c^{-2})\rceil, c^\alpha\big)+1.
	\]
	Thus, if we let
	\[
	h\colon\bR\times(0,2^{-1}]\to \bR\times(0,2^{-1}],\quad h(x,y):=(x+\alpha\log y,y^\alpha),
	\]
	we find
	\[
	f(n,c)\geq\max\{m\mid \textrm{$x$-coordinate\ of\ }h^m(n,c)\geq1\}
	\]
	where $h^m$ denotes $m$-fold composition. Since $h^m(x,y)=(x+\frac{\alpha^{m+1}-\alpha}{\alpha-1}\log y,y^{\alpha^m})$, we see
	\[
	f(n,c)\geq \bigg\lfloor\log_\alpha\bigg(\frac{(1-n)(\alpha-1)}{\alpha\log c}+1\bigg)\bigg\rfloor.
	\]
	For $n\geq2$, we have $\log_\alpha(\frac{n}{n-1})\leq\log_\alpha 2$, and so
	\[
	f(n,c)\geq\log_\alpha\bigg(\frac{(n-1)(\alpha-1)}{\alpha\log(c^{-1})}\bigg)-1\geq\log_\alpha n-\beta(c)
	\]
	where
	\[
	\beta(c):=\log_\alpha 2 + \log_\alpha\log(c^{-1}) - \log_\alpha(\alpha-1) + 2
	\]
	is a constant depending only on $c$.
	
	Next,
	\[
	\log_\alpha n = \frac{\log_2 n}{\log_2(2+\frac{\epsilon}{3})}=\frac{\log_2 n}{1+\log_2(1+\frac{\epsilon}{6})}\geq \left(1-\frac{\epsilon}{2}\right)\log_2 n,
	\]
	where the last inequality uses \cref{eqn:epsilon-suff-small-liminf}. Hence,
	\[
	f(n,c)\geq \left(1-\frac{\epsilon}{2}\right)\log_2 n-\beta(c)\geq (1-\epsilon)\log_2 n
	\]
	for $n$ sufficiently large.
\end{proof}

\section{Proof of the main theorems}

First, we quickly prove \cref{prop:bounded M and M_0}.

\begin{proof}[{Proof of \cref{prop:bounded M and M_0}}]
	The reverse direction follows from \cref{rmk:compare M and M_0}. Let $L(-)$ denote the full Minkowski length as in \cite[Definition 4.5]{dolorfino2022good-families-toric}. If $\{M(P_i)\}_i$ is bounded, then \cite[Proposition 4.7]{dolorfino2022good-families-toric} shows that $\{L(P_i)\}_i$ is bounded. Fix $i$ and let $m:=M_0(P_i\cap \Z^{n_i})$. Since $P_i$ is convex, there exists an injective integer affine transformation $F$ such that $F([0,1]^m)\subseteq P_i$. By definition, $L(P)\leq L(Q)$ whenever $P\subseteq Q$, and $L([0,1]^m)=m$, so we have
	\begin{align*}
		L(P_i) \geq L(F([0,1]^m)) \geq L([0,1]^m) = m = M_0(P_i\cap \Z^{n_i})
	\end{align*}
	which implies that $\{M_0(P_i\cap \Z^{n_i})\}_i$ is bounded. 
\end{proof}

Next, having given a lower bound for $f_N$ in \cref{prop:main-lower-bound}, we can complete the proof of \cref{thm:main} by giving the following upper bound.

\begin{proposition}\label{prop:main-upper-bound}
	Let $c\in (0,1]$, let $N\geq2$ be an integer, and let $\epsilon>0$. Then there exists a family of sets $S_n\subseteq[N]^n$ such that $\liminf_n\delta(S_n)\ge c$ and, for sufficiently large $n$, $M_0(S_n) \le (1+\epsilon)\log_2(n)$.
\end{proposition}
\begin{proof}
	It suffices to prove this in the case when $c=1$. For $n\ge 3$, we let $c_n = 1-1/N^{\lfloor \log(\log(n))\rfloor}$. Observe that for $n$ sufficiently large we have
	\begin{equation}\log(4) + n((1+\epsilon) \log_2(n)+1) < n^{1+\epsilon/2} \log(1/c_n),
		\label{eq:ep}\end{equation} since 
	$$\log(1/c_n)>1/(2N^{\lfloor \log \log(n)\rfloor})> 1/(2N^{2\log(\log(n))}) > 1/(2\log(n)^2)$$ for $n$ large.
	
	We let $n\ge 3$ be such that the inequality in \cref{eq:ep} holds and
	we pick a positive integer $r$ in the interval $((1+\epsilon/2)\log_2(n),(1+\epsilon)\log_2(n))$.

	Observe that an injective integer affine map from $\{0,1\}^r$ to $[N]^n$ is uniquely determined by the images of $e_1,\ldots, e_r, e_1+\cdots +e_r$, where $e_i\in \{0,1\}^r$ denotes the vector with a $1$ in the $i$-th coordinate and zeros in every other coordinate.  Since there are at most $N^n$ choices for each vectors, we see there are at most $N^{n(r+1)}$ injective integer affine maps from $\{0,1\}^r$ to $[N]^n$.  We let
	$Q_1,\ldots ,Q_L\subseteq [N]^n$ with $L\le N^{n(r+1)}$ denote the distinct images of these injective maps.  
	
	We now consider events $X_1,\ldots ,X_L$, where $X_i$ is the event that a subset of $[N]^n$ of density $c'$ contains the set $Q_i$. The set $Q_i$ has size $2^r$ and so the probability that a subset of $[N]^n$ of density $c_n$ contains $Q_i$ is given by
	$${N^n -2^r \choose N^n c_n - 2^r}{N^n \choose N^n c_n}^{-1} = \prod_{i=0}^{2^r-1} \frac{N^n c_n-i}{N^n-i}.$$
	Since $c_n<1$, we see that $(N^n c_n -i)(N^n-i)^{-1} < c_n$ for $i=0,\ldots ,2^r-1$, and so the event $X_i$ occurs with probability at most $c_n^{2^r}$.  
	By the Lovasz Local Lemma \cite[Lecture 8]{Spencerbook}, there is a nonzero chance that none of the events $X_1,\ldots ,X_L$ occur, provided $4L\cdot c_n^{2^r} < 1$ and since $L\le N^{n(r+1)}$, it is enough to have the inequality $\log(4) + n(r+1) < 2^r \log(1/c_n)$.  Since $(1+\epsilon/2)\log_2(n)<r<(1+\epsilon)\log_2(n)$, we then see it is sufficient to have the inequality
	$$\log(4) + n((1+\epsilon) \log_2(n)+1) < n^{1+\epsilon/2} \log(1/c_n),$$ which holds for $n$ large by \cref{eq:ep}.
	
	It follows that for $n$ sufficiently large there exists a set $S\subseteq [N]^n$ of density $c_n$ that contains none of $Q_1,\ldots ,Q_L$ and hence $M_0(S)<r < (1+\epsilon)\log_2(n)$.  The result follows.
\end{proof}

\begin{proof}[{Proof of \cref{thm:main}}]
	The result follows immediately from the lower bound given in \cref{prop:main-lower-bound} and the upper bound given in \cref{prop:main-upper-bound}.
\end{proof}

Lastly, we prove \cref{thm:c=0}.

\begin{proof}[Proof of \cref{thm:c=0}] 
	Notice that if we have an infinite subset $T$ of $\mathbb{N}$ and we have a family of nonempty sets $\{S_n\}_{n\in T}$ with $S_n\subseteq \{0,\ldots ,N-1\}^n$ such that $\lim_n \log(|S_n|)/n=1$ then if we let $c_n=\delta(S_n)$, we have $c_n^{1/n}\to 1$. Thus it suffices to show that $f(n,c_n)\to \infty$ as $n\to\infty$ whenever $c_n\in (0,1]$ is a sequence of positive numbers with $c_n^{1/n}\to 1$.  
	
	To show this, suppose towards a contradiction that there exists an infinite subset $T$ of $\mathbb{N}$ and a sequence of positive real numbers $\{c_n\}_{n\in T}$ with $c_n^{1/n}\to 1$ such that $\liminf_n f(n,c_n)<\infty$.  Then we may select such a sequence $\{c_n\}$ with $\liminf_n f(n,c_n) = m<\infty$, with $m$ minimal among all such sequences.
	
	Let $\epsilon>0$. Then since $c_n^{1/n}\to 1$, we have $c_n>q^{-\epsilon n}$ for $n$ sufficiently large. In particular, for $n$ sufficiently large we have
	$\log(8c_n^{-2}) < \log(8) + 2\epsilon n<n$, and so by \cref{prop:inductive-f-bound},
	\[
	f(n,c_n) \ge f\big(n-\lceil\log(8c_n^{-2})\rceil, 2c_n^2(c_n+4)^{-2}\big)+1.
	\]
	We now let $b_n = 2c_n^2/(c_n+4)^2$. Then $b_n^{1/n}\to 1$ and since $n-\lceil\log(8c_n^{-2})\rceil\to\infty$ as $n\to\infty$, we see by minimality of $m$ that
	$f\big(n-\lceil\log(8c_n^{-2})\rceil, 2c_n^2(c_n+4)^{-2}\big)\ge m$ for $n$ sufficiently large.  But this now gives that $f(n,c_n)\ge m+1$ for $n$ sufficiently large, a contradiction.  It follows that $f(n,c_n)\to\infty$ whenever $c_n^{1/n}\to 1$, giving the first part of \cref{thm:c=0}.
	
	To prove the assertion about optimality, we again use the Lovasz Local Lemma. We fix $\epsilon>0$ and pick $r$ such that $2^{r-1}/(r+3) > 1/\epsilon$. As in the proof of \cref{prop:main-upper-bound}, there are at most 
	$N^{n(r+1)}$ injective integer affine maps from $\{0,1\}^r$ to $[N]^n$, and each one has image of size $2^r$. We let
	$Q_1,\ldots ,Q_L\subseteq [N]^n$ with $L\le N^{n(r+1)}$ denote the distinct images of these injective maps and let $X_i$ be event that a subset of $[N]^n$ of density $c_n:=N^{-\lfloor \epsilon n\rfloor}$ contains the set $Q_i$ for $i=1,\ldots ,L$. 
	
	As in the proof of \cref{prop:main-upper-bound}, the event $X_i$ occurs with probability at most $c_n^{2^r}$. To show that there is a nonzero probability that no $X_i$ occurs, by the Lovasz Local Lemma it suffices to show that 
	$4 L c_n^{2^r}<1$.  Using the fact that $L\le N^{n(r+1)}$, we see for $n$ sufficiently large we have
	\begin{align*} 4Lc_n^{2^r} &\le 4\cdot N^{n(r+1)} N^{-2^r \lfloor \epsilon n\rfloor}\\
		&\le 4\cdot N^{n(r+1)} \cdot N^{-2^r (\epsilon n -1)} \\
		&\le N^2 \cdot N^{n(r+1) - 2^r \epsilon n/2}\\
		&\le N^{n(r+3) - 2^{r-1}\epsilon n},
	\end{align*}
	which is strictly smaller than $1$ by our choice of $r$.  It follows that for $n$ large, there is a subset $S_n$ of $\{0,\ldots ,N-1\}^n$ with density at least $N^{-\lfloor \epsilon n\rfloor}$ such that $M_0(S_n)\le r$.  In particular, $|S_n| \ge N^{(1-\epsilon)n}$, completing the proof.
\end{proof}

\printbibliography

\end{document}